\documentclass[10pt,twoside,reqno]{amsart}
\usepackage[all]{xy}
        \usepackage {amssymb,latexsym,amsthm,amsmath,mathtools,dsfont,mathrsfs}
        \usepackage{enumitem,color}

\usepackage[
  top=1in,
  bottom=1in,
  left=1.2in,
  right=1.2in
]{geometry}
\usepackage{makecell}
\usepackage{mathtools}

\usepackage{longtable}

\usepackage[backend=bibtex,backref=true,maxbibnames=999,doi=false,isbn=false,url=false
]{biblatex} 
\usepackage{float}
\usepackage{hyperref}
\long\def\symbolfootnote[#1]#2{\begingroup%
\def\thefootnote{\fnsymbol{footnote}}\footnote[#1]{#2}\endgroup}

\makeatletter
\def\imod#1{\allowbreak\mkern10mu({\operator@font mod}\,\,#1)}
\makeatother
\makeatletter
\renewcommand*\env@matrix[1][*\c@MaxMatrixCols c]{%
  \hskip -\arraycolsep
  \let\@ifnextchar\new@ifnextchar
  \array{#1}}
\makeatother
\usepackage[capitalize,nameinlink]{cleveref} 
\usepackage{comment} 
\usepackage{xcolor} 
\hypersetup{
    colorlinks,
    linkcolor={red!80!black},
    citecolor={green!80!black},
    urlcolor={blue!80!black}
}

\newtheorem{theorem}{Theorem}[section]
\newtheorem{lemma}[theorem]{Lemma}
\newtheorem{corollary}[theorem]{Corollary}
\newtheorem{proposition}[theorem]{Proposition}
\newtheorem*{theorem*}{Theorem}
\theoremstyle{definition}

\newtheorem{remark}[theorem]{Remark}
\newtheorem{example}[theorem]{Example}

\numberwithin{equation}{section}
\newcommand{\ignore}[1]{}

\newcommand{\mynote}[1]{}
\newcommand{\Spec}{\operatorname{Spec}}

\title[Commuting graph of non-abelian groups]{Commuting graph of non-abelian groups of order $p^n$ with center having $p^{n-2}$ elements}
\author[Kaur D.]{Dilpreet Kaur}
\email{dilpreetkaur@iitj.ac.in}
\address{Indian Institute of Technology Jodhpur
N.H. 62, Nagaur Road, Karwar Jodhpur 342030
Rajasthan}
\author[Kumar G.]{Gaurav Kumar}
\email{p24ma0001@iitj.ac.in}
\address{Indian Institute of Technology Jodhpur
N.H. 62, Nagaur Road, Karwar Jodhpur 342030
Rajasthan}
\author[Tiwari S.]{Shivam Tiwari}
\email{m24ma1018@iitj.ac.in}
\address{Indian Institute of Technology Jodhpur
N.H. 62, Nagaur Road, Karwar Jodhpur 342030
Rajasthan}
\thanks{}
\date{\today}
\subjclass[2020]{05C25, 20D15}
\keywords{Commuting Graph, $p$-groups, Centralizers}
\begin{document}
\setcounter{section}{0}
\begin{abstract}
Let $G$ be a group. The commuting graph $\zeta(G, V)$ of a group $G$ has a 
vertex set $V\subseteq G,$ two vertices are connected with an edge if the corresponding elements commute in $G.$ In this article, we study the commuting graphs for the non-abelian $p$-groups of order $p^n$ with center size $p^{n-2}.$ We study detour distance, metric dimension, resolving polynomials, and the spectral properties for these graphs.
\end{abstract}
\maketitle
\section{Introduction}\label{Section.intro}

For a group $G,$ there are various graphs defined on it, which capture the structure of the group from different aspects.  One of the most commonly studied graphs on a group is the commuting graph. Let $G$ be a group. The commuting graph $\zeta(G, V)$ of a group $G$ has set of vertices $V\subseteq G,$ and two vertices in graph are connected if and only if the corresponding elements commute in the group $G.$  

The commuting graph of a group was first used by Brauer and Flower in \cite{BF} to study finite groups of even order. This was one of the early step towards the classification of simple groups. Later, Fischer also used related concepts in the discovery of Fischer groups in \cite{Fischer}. However they did not define the commuting graph formally. Bertram formally defined and used the term commuting graphs in his article\cite{BE}. 
Since then, commuting graphs has been studied by many authors. In \cite{Ali02062016}, authors study the commuting graphs defined on dihedral groups. In \cite{kakkar_gen} and \cite{kumar2021commuting}, authors advanced the study of commuting graphs by studying commuting graphs defined on generalized dihedral groups, and semi-dihedral groups, respectively. The commuting graphs defined on symmetric groups and alternating groups has been studied in \cite{Iranmanesh2008} and \cite{Bates2009}. In \cite {Arvind2025}, computational problem of deciding whether a given graph is the commuting graph of a finite group has been discussed. In \cite{SV}, commuting graphs of non-abelian $p$-groups of order $p^4$ with center size $p$ are studied, where $p$ is an odd prime number. We refer the interested reader to see the survey article \cite{cameron2021graphs}.

The aim of this article is to study the commuting graph of non-abelian groups of order $p^n$ with center size $p^{n-2},$ where $p$ is an odd prime number. 
This article is organized as follows: In section \ref{section.partition}, we discuss the basic properties of groups of order $p^{n}$ with center containing $p^{n-2}$ elements. In the next section \ref{Section.commuting}, we compute the graph, the degree of its vertices, its number of edges, and chromatic number. In section \ref{section.Detour},
 we study the detour distance properties of commuting graphs of the non-abelian $p$-group of order $p^{n}$ with center containing $p^{n-2}$ elements. We also compute its clique number and its independence number. The section \ref{section.polynomial} is devoted to the study of metric dimensions and resolving polynomials of such graphs. We conclude the article with the study of spectral properties of commuting graph of these $p$-groups, such as spectral properties of signless Laplacian matrices and normalized Laplacian matrices.
\section{Partition of groups of order $p^n$ with center having $p^{n-2}$ elements.}\label{section.partition}
In this section, $p$ denotes an odd prime number. Let $G$ be a group and $Z(G)$ denotes the center of group $G.$ For any $g\in G,~~C_G(g)$ denotes the centralizer of $g$ in $G.$
\begin{lemma}\label{centralizer_size}
Let $G$ be a group such that $|G|=p^n$ and $|Z(G)|=p^{n-2}.$ If $g\in G\backslash  Z(G),$ then $|C_G(g)|=p^{n-1}.$ Moreover $C_G(g)=\langle g, Z(G) \rangle.$
\end{lemma}
\begin{proof}
    Clearly $Z(G) \subset C_G(g) \subset G$ for all $g \in G.$ In this case $Z(G)\neq C_G(g)$ and $C_G(g)\neq G$ as $g \in C_G(g)$ and $g\notin Z(G).$ Also as $C_G(g)$ is a subgroup of $G,$ $|C_G(g)|$ divides $p^n.$ This forces $C_G(g)=p^{n-1}.$

    Now notice that $\langle g, Z(G) \rangle \subseteq C_G(g).$ By same arguments as above, we get $|\langle g, Z(G) \rangle|=p^{n-1}$ and hence $C_G(g)=\langle g, Z(G) \rangle.$
\end{proof}

\begin{lemma}\label{centralizer_intersection}
Let $G$ be a group such that $|G|=p^n$ and $|Z(G)|=p^{n-2}.$ If $g_1, g_2\in G\backslash Z(G),$ then either $C_G(g_1)=C_G(g_2)$ or $C_G(g_1)\cap C_G(g_2)=Z(G).$
\end{lemma}
\begin{proof}
We know that $Z(G) \subseteq C_G(g_1) \cap C_G(g_2)$ for all $g_1, g_2 \in G.$ Let $h \in (C_G(g_1) \cap C_G(g_2))\backslash Z(G).$ This implies $hZ(G) \in \frac{C_G(g_i)}{Z(G)}$ for $i=1,2.$ Now using Lemma \ref{centralizer_size}, we get $|\frac{C_G(g_i)}{Z(G)}|=p, $ for $i=1,2.$ As $hZ(G)\neq Z(G), $ we get $\frac{C_G(g_i)}{Z(G)}=\langle hZ(G) \rangle$ for $i=1,2.$ Hence $C_G(g_i)=\langle h, Z(G)\rangle $ for $i=1,2.$
\end{proof}

Let $G$ be a $p$-group of size $p^{n}$ and $|Z(G)|=p^{n-2}.$
We partition the group $G$ into two classes as follows:

\begin{enumerate}
    \item $\Omega_1 =Z(G).$ As mentioned above $|Z(G)|=p^{n-2}.$
    \item $\Omega_2 =G\backslash Z(G).$ The size of $\Omega_2 =p^n-p^{n-2}.$ It follows from Lemma \ref{centralizer_size}, $|C_G(g)|=p^{n-1}$ for all $g\in \Omega_2.$ Also it follows from Lemma \ref{centralizer_intersection}, the class $\Omega_2$ of $p^n-p^{n-2}$ elements is partitioned into $p+1$ blocks of $p^{n-1}-p^{n-2}$ elements. We denote these blocks by $B_1, B_2, \dots,B_{p+1}.$ Further, it follows from Lemma \ref{centralizer_intersection}, no element of block $B_i$  commutes with any element of block $B_j$ if $i\neq j.$ 
    \end{enumerate}
    We compute $|G|=|\Omega_1|+|\Omega_2|=p^{n-2}+(p+1)(p^{n-1}-p^{n-2})=p^n.$
\section{Commuting graph and its properties}\label{Section.commuting}
In this section, we first compute the commuting graph of group $G$ of size $p^n$ with center $Z(G)$ of size $p^{n-2},$ for an odd prime number $p.$ We also investigate some properties of this graph. We recall the commuting graph $\zeta(G,V)$ of $G$ with vertex set $V \subseteq G,$ and two vertices are connected with an edge if the corresponding elements of group commute with each other. We denote the complete graph on $n$ vertices with $K_n.$ The join of two graphs $\Gamma_1(V_1,E_1)$ and $\Gamma_2(V_2,E_2)$ with disjoint vertex sets, is defined as the graph obtained by connecting all the vertices of $\Gamma_1$ to all the vertices of $\Gamma_2$ and it is denoted by $\Gamma_1\vee\Gamma_2.$

\begin{proposition} \label{structure of graph}
 Using same notations as the previous section,  if $V$ is a subset of $G,$ then 

  $$\zeta(G,V) =\begin{cases}
      K_{p^{n-2}}, & V=\Omega_1,\\
      K_{p^{n-1}-p^{n-2}}, & V=B_k, (1\leq k\leq p+1).
  \end{cases}$$

  Moreover $\zeta(G,G)=K_{p^{n-2}} \vee (p+1).K_{p^{n-1}-p^{n-2}}.$
\end{proposition}
\begin{proof}
    Using Lemma \ref{centralizer_intersection}, we get there is no edge connecting any two vertices belonging to distinct blocks $B_k.$ Recall $\Omega_2=G\setminus Z(G),$ therefore 
    $$\zeta(G, \Omega_2)= \bigcup_{k=1}^{p+1} \zeta (G, B_k)= (p+1).K_{p^{n-1}-p^{n-2}}.$$
    As $\Omega_1$ is the center of group $G,$ all the vertices of set $\Omega_1$ are adjacent to all the vertices in $\Omega_1$ as well as to all the vertices in $\Omega_2.$ This implies $\zeta(G,G)=K_{p^{n-1}} \vee (p+1).K_{p^{n-1}-p^{n-2}}.$
 \end{proof}
 \begin{corollary}\label{degree_of_vertices}
    With same notations as above,  let $u$ be a vertex of graph $\zeta(G,G).$ Then
    $$\deg{u} =\begin{cases}
    p^{n}-1, & u \in \Omega_1,\\
    p^{n-1}-1, & u \in \Omega_2.
    \end{cases}$$
 \end{corollary}
 \begin{corollary}
     The number of edges in graph $\zeta(G,G)$ is given by
     $$\displaystyle E(\zeta(G,G))= \frac{p^{2n-1}+p^{2n-2}-p^{2n-3}-p^n}{2}$$ 
 \end{corollary}
 \begin{proof}
     Using the Handshaking Lemma, we see that twice the number of edges in a graph is the same as sum of degrees of all the vertices. We use the corollary \ref{degree_of_vertices}, to count the number of edges in the graph $\zeta(G,G).$ 
 \end{proof}
 The chromatic number of a graph is the minimum number of colors required to color all the vertices such that no two adjacent vertices have the same color. 

 \begin{proposition}
     The chromatic number of $\zeta(G,G)$ is $p^{n-1}.$
 \end{proposition}
 \begin{proof}
 We first color the vertices of $\Omega_1.$ As $\zeta(G, \Omega_1)$ is a complete graph $K_{p^{n-2}}.$ We need $p^{n-2}$ colors to color the graph $\zeta(G, \Omega_1).$ The coloring of $\zeta(G, \Omega_2)$ is done block-wise, where each block $B_k$ for $1\leq k \leq p+1$ has $p^{n-1}-p^{n-2}$ vertices. All the vertices in block $B_1$ are adjacent to all the vertices in $\Omega_1,$ so we need $p^{n-1}-p^{n-2}$ more colors, different from the colors used to color $\zeta(G, \Omega_1).$ As no two distinct blocks share a common edge, the same colors as used to color $B_1$ can be used to color all blocks $B_k$ for all $2\leq k\leq p+1.$ Hence, the chromatic number $\zeta(G,G)$ is $p^{n-2}+(p^{n-1}-p^{n-2})=p^{n-1}.$
 \end{proof} 
\section{Detour distance of the commuting graph }\label{section.Detour}
Let $\Gamma(V, E)$ be a connected graph and let $u,w\in V.$ The detour distance $D(u,w) $ between 
$u$ and $w$ is defined as the length of the longest path between them. The detour eccentricity $ecc_D(u)$ of a vertex $u\in V$ is defined as $max\{D(u,w) ~|~ w\in V\}.$
Further, the detour radius $rad_d(G)$ of the graph $\Gamma(V, E)$ is defined as $min \{ ecc_D(u) ~|~ u \in V\}, $ the detour diameter $diam_d(G)$ is defined as $max \{ ecc_D(u) ~|~ u \in V\}.$
\begin{theorem}
Let $u$ be a vertex of graph $\zeta(G,G).$ Then the detour eccentricity is as given below:
$$ \displaystyle  ecc_D(u)=\begin{cases}
 p^3-p^2+p-1 & \text{ when } n=3 \text{ and } u\in \Omega_1,\\
 p^n-1 & \text{ otherwise. }   
\end{cases}$$
Moreover $rad_D(\zeta(G,G))=p^3-p^2+p-1$ and $diam_D(\zeta(G,G))=p^3-1$ if $n=3.$ For all $n>3,$  we have $rad_D(\zeta(G,G))=diam_D(\zeta(G,G))=p^n-1.$
\end{theorem}

\begin{proof}
    
    We first assume that $n=3.$ In this case, $|\Omega_1|=p$ is less than the number of blocks, which is $p+1.$ Using notations of Prop. \ref{structure of graph}, let $\Omega_1= \{ u_1, u_2, \dots, u_p\}$  and let $B_1, B_2 \dots, B_{p+1}$ be distinct blocks.
    Without loss of generality,  we choose $u_1\in \Omega_1.$  To compute $ecc_D(u_1),$ we begin from $u_1,$ and traverse each vertex in the block $B_1$ without repetition and come back in $\Omega_1$ at a vertex $u_2.$ We can repeat this process $(p-1)$ times. After traversing the vertices of the $p^{th}$ block $B_p$, we do not have any vertex left in $\Omega_1$ to come back in $\Omega_1.$ This implies that the detour eccentricity of $u$ is 
    $(p-1)(p^2-p+1)+(p^2-p)=p^3-p^2+p-1.$ 
    
    Next, without loss of generality, suppose we choose $u\in B_1.$  To compute $ecc_D(u),$ we first traverse each vertex in block $B_1$ without repetition, then we go to vertex $u_1$ in $\Omega_1.$ Now we go to block $B_2.$ We repeat this process until we have traversed all the blocks. Note that in this case, we traverse each vertex of the graph. Hence we compute that the detour eccentricity in this case is 
    $p(p^2-p+1)+(p^2-p)=p^3-1.$
    
    Now we assume that $n>3.$ In this case, $|\Omega_1|=p^{n-2}$ is greater than the number of blocks, which is $p+1.$ Again without loss of generality, suppose we choose $u_1\in \Omega_1.$  To compute $ecc_D(u_1),$ we begin from $u_1,$ and traverse each vertex in the block $B_1$ without repetition and come back in $\Omega_1$ at  vertex $u_2.$ We can repeat this process until we have traversed all the blocks. After traversing all blocks, we can traverse the remaining elements in the set $\Omega_1.$ Similarly, if we start from $u\in B_k,$ for some $1\leq k\leq p+1,$ we can traverse all the vertices of the graph in this case as well. Hence, in this case, we get that the detour eccentricity is $p^n-1$ for all vertices of the graph.

    The detour radius and the detour diameter directly follow from the calculation of the detour eccentricity of all the vertices of the graph.
\end{proof}

For a graph $\Gamma(V,E),$ the clique is a subset of vertices in which every pair of vertices is connected by an edge. Thus, induced subgraph on this subset of vertices is a complete graph. The clique number of a graph $\Gamma(V,E)$ is defined as the number of vertices in the maximum clique in $\Gamma(V,E).$
\begin{theorem} 
The clique number of the graph $\zeta(G,G)$ is $p^{n-1}.$
\end{theorem}
\begin{proof}
    The given graph is $\zeta(G,G) =K_{p^{n-1}} \vee (p+1).K_{p^{n-1}-p^{n-2}}.$ The subgraph $K_{p^{n-2}}$ is defined on the central elements of the groups, while the $(p+1)$ subgraphs $K_{p^{n-1}-p^{n-2}}$ are defined on the sets containing non-central elements of distinct centralizers. Therefore, the maximum clique is formed on elements of the set $\Omega_1 \cup B_k$ for some $k, ~ 1\leq k \leq p+1.$ This implies that the clique number of the graph $\zeta(G,G)$ is $ p^{n-2}+(p^{n-1}-p^{n-2})=p^{n-1}.$
\end{proof}

The conceptual opposite of a clique is an independent set. For a graph $\Gamma(V,E),$ the independent set is a subset of vertices in which no two of vertices are connected with an edge. The number of vertices in the maximum independent set of a graph is known as its independence number.

\begin{theorem}
The independence number of the graph $\zeta(G,G)$ is $p+1.$
\end{theorem}
\begin{proof}
    Since elements in the center of any group are connected with all other elements of the group, so no element from the center belongs to the independent set of the commuting graph. Since centralizers of $p-$ groups with order $p^n$ and center size $p^{n-2}$ are abelian, therefore all  elements in the same centralizer also commute with each other, so we can choose only one element from each centralizer. Since there are a total $p+1$ distinct centralizers, taking one non-central element from each centralizer to form maximum independent set, we find that the independence number of the graph $\zeta(G,G)$ is $p+1.$ 
\end{proof}

\section{Resolving polynomial of the commuting graph}\label{section.polynomial}
Let $\Gamma(V, E)$ denote a graph with the set of vertices $V$ and the set of edges $E.$
For a vertex $v$ in a graph $\Gamma(V,E)$, the open neighborhood of $v$ is the set of all vertices adjacent to $v$ and it is denoted by $N(v).$ The closed neighborhood of $v$ includes $v$ itself along with all its neighbors and is denoted by $N[v],$ and hence $N[v]=N(v)~\cup~\{v\}.$ For a graph $\Gamma(V,E),$ a subset $T$ of $V$ is called a twin set if every two distinct vertices $u,v\in T$ satisfy $N[u]=N[v]$ or $N(u)=N(v),$ which means every pair of vertices in $T$ has exactly same neighborhood. \\

For a graph $\Gamma(V,E),$ a subset $S\subseteq V$ is called a resolving set of $\Gamma(V,E)$ if for every pair of distinct vertices $u,v\in V$ there exists at least one vertex $w\in S$ such that $d(u,w)\neq d(v,w),$ where $d(u,v)$ denotes the length of the shortest path between the vertices $u$ and 
$v$. Minimum cardinality of the resolving set is called metric dimension of graph $\Gamma(V,E)$ and it is denoted by $\beta(\Gamma).$ The resolving polynomial of $\Gamma(V,E)$ encodes the number of resolving sets of every possible cardinality defined as $\beta(\Gamma,x)=\sum\limits_{i=\beta(\Gamma)}^{n} s_ix^i,$ where $i$ is size of resolving set, $s_i$ is number of resolving set of that size and $n$ is the total number of vertices.\\

\begin{remark}
\cite[Remark 3.3]{Ali02062016} \label{5.1} 
   If $U$ is a twin set in a connected graph $\Gamma(V,E)$ of order $n$ with $|U|=l\geq 2,$ then every resolving set for $G$ contains at least $l-1$ vertices of $U.$
   \end{remark}
We recall that $\zeta(G,G)$ denotes the commuting graph of group $G$ of size $p^n$ with center of size $p^{n-2}.$
    \begin{theorem} \label{Metric_dimension of graph}
        The metric dimension of the graph $\Gamma=\zeta(G,G)$ is $p^n-(p+2).$ 
    \end{theorem}
\begin{proof}
    With same notations as Prop. \ref{structure of graph}, we have  $\zeta(G,G)=K_{p^{n-2}} \vee (p+1).K_{p^{n-1}-p^{n-2}}.$ Recall that $|\Omega_1|=p^{n-2}.$ Note that $N[u]=N[v]=G~\forall~u,v\in \Omega_1$ and there does not exist any $w\in G\setminus\Omega_1$ such that $N[w]=G.$ Hence $\Omega_1$ is a twin set in $\Gamma.$ Similarly each block $B_k,1\leq k \leq p+1$ is a twin set with $|B_k|=p^{n-1}-p^{n-2}.$ Hence all twin sets in graph $\Gamma=\zeta(G,G)$ are $\Omega_1,B
    _1,B_2,...,B_{p+1}.$\\
    Therefore by Remark \ref{5.1} , $\beta(\Gamma)\geq (p^{n-2}-1)+(p+1)(p^{n-1}-p^{n-2}-1)= p^n-(p+2).$\\
    As we have total $p+2$ twin sets. To construct a resolving set, we include all vertices from each twin set except exactly one.
    Let $S$ denotes the resolving set, and we define 
    \begin{align*}
       S\cap \Omega_1 &= \big\{v_1,v_2,...,v_{p^{n-2}-1}\big\},\\
    S\cap B_1 &=\big\{w_1^1,w_2^1,...,w_{p^{n-1}-p^{n-2}-1}^1\big\},\\
    \vdots\\
    S\cap B_{p+1} &=\big\{w_1^{p+1},w_2^{p+1},...,w_{p^{n-1}-p^{n-2}-1}^{p+1}\big\},
    \end{align*}
     where $v_i\in \Omega_1,w_k^j\in B_j~\text{for}~ 1\leq i\leq p^{n-2}-1,~ 1\leq j \leq p+1~ \text{and}~ 1\leq k\leq p^{n-1}-p^{n-2}-1.$
    
    Hence $|S|=(p^{n-2}-1)+(p+1)(p^{n-1}-p^{n-2}-1)=p^n-(p+2).$ Now, to resolve any pair $(a,b)\in \zeta(G,G),$ with $a\neq b.$ We have four cases:
    \begin{enumerate}
\renewcommand{\labelenumi}{(\roman{enumi})}
\item $a,b\in\Omega_1$ 
\item $a\in \Omega_1, b\in B_k$
\item $a,b\in B_i$
\item $a\in B_i,b\in B_j, ~ i\neq j.$
\end{enumerate}
In the first and the third case, note that at least one of $a$ and $b$ belongs to the resolving set, now  we use $d(a,b)=1, d(a,a)=d(b,b)=0.$ In the second case, for some $j\neq k,$ we have $d(a,\omega_l^j)=1,~d(\omega_l^j,b)=2.$ In the fourth case, as $i\neq j,$ we have $d(a,\omega_l^i)=1,~d(b,\omega_l^i)=2.$ This implies $S$ is a resolving set of size $p^n-(p+2).$ Therefore, $\beta(\Gamma)\leq|S|=p^n-(p+2).$ 
    
\end{proof}

\begin{theorem}
    Let $\Gamma=\zeta(G,G),$ then the resolving polynomial of $\Gamma$ is $$\beta(\Gamma,x)=x^{p^n-(p+2)}[p^{n-2}(p^{n-1}-p^{n-2})^{p+1}+x(p^{n-1}-p^{n-2})^{p}2p^{n-1}]+\sum\limits_{i=p^n-p}^{p^n}s_ix^i,$$ where $s_i={}^{p+1}C_{i-p^n+(p+2)}(p^{n-1}-p^{n-2})^{p^n-i}+(p^{n-2}){}^{p+1}C_{i-p^n+(p+2)}(p^{n-1}-p^{n-2})^{p^n-i-1}.$
\end{theorem}
\begin{proof}
    We begin by computing $s_j$ for $j=p^n-(p+2).$ Since all the twin sets in the graph $\Gamma$ are $\Omega_1,B_1,B_2,...,B_{p+1}.$ Total number of ways to form resolving sets of size $j$ is given by
    
    $$s_j={}^{p^{n-2}}C_{p^{n-2}-1}\big({}^{p^{n-1}-p^{n-2}}C_{p^{n-1}-p^{n-2}-1}\big)^{p+1}=p^{n-2}(p^{n-1}-p^{n-2})^{p+1}.$$

    Now, we determine the number of resolving sets of cardinality $p^n-(p+2)+t$ for $t\geq1,$ by including $t$ more vertices in the resolving set of size $\beta(\Gamma)$. For $t=1,$ we have two ways:
\begin{enumerate}
\renewcommand{\labelenumi}{(\roman{enumi})}
\item take new vertex from $\Omega_1$ 
\item take new vertex from $B_k,$ for some $k.$
\end{enumerate}

In first case, we take all vertices from $\Omega_1,$ so we have $\big({}^{p^{n-1}-p^{n-2}}C_{p^{n-1}-p^{n-2}-1}\big)^{p+1}=(p^{n-1}-p^{n-2})^{p+1}$ choices for such resolving sets.

In second case, there are $(p+1)$ blocks to choose from, therefore we have $(p^{n-2})(p^{n-1}-p^{n-2})^{p}(p+1)$ choices. This implies
$$s_{p^n-(p+2)+1}=(p^{n-1}-p^{n-2})^{p+1}+(p^{n-2})(p^{n-1}-p^{n-2})^p(p+1)=(p^{n-1}-p^{n-2})^p(2p^{n-1}).$$

Now we determine $s_{p^4-(p+2)+t}$ for $2\leq t\leq(p+2).$ There are again two ways to include $t$ more vertices in the resolving set of size $\beta(\Gamma).$ 
\begin{enumerate}
    \item one vertex from $\Omega_1$ and rest of $(t-1)$ from $\Omega_2$ 
    \item all $t$ vertices from  $\Omega_2$ 
\end{enumerate}

    In case (1) we get
    $ {}^{(p+1)}C_{t-1}.(p^{n-1}-p^{n-2})^{p+2-t}$ choices of resolving sets of size $p^4-(p+2)+t,~2\leq t\leq p+2. $ \\
    Similarly, in case (2),  we get $p^{n-2}. {}^{p+1}C_{t}.(p^{n-1}-p^{n-2})^{p+1-t},$ choices for resolving sets.\\
    This implies, for $i=\beta(\Gamma)+t,$ we get $$s_i={}^{p+1}C_{t-1}(p^{n-1}-p^{n-2})^{p+2-t}+p^{n-2}{}^{p+1}C_t(p^{n-1}-p^{n-2})^{p+1-t}.$$

     Further simplifying, we get $\beta(\Gamma,x)=x^{p^n-(p+2)}[p^{n-2}(p^{n-1}-p^{n-2})^{p+1}+x(p^{n-1}-p^{n-2})^{p}.2p^{n-1}]+
     \sum\limits_{i=p^n-p}^{p^n}\Big({}^{p+1}C_{i-p^n+(p+2)}(p^{n-1}-p^{n-2})^{p^n-i}+(p^{n-2}){}^{p+1}C_{i-p^n+(p+2)}(p^{n-1}-p^{n-2})^{p^n-i-1}\Big)x^i.$

    \end{proof}
\section{Spectral Properties of the commuting graph}
\label{section.spectral}
If $\lambda_1,\lambda_2,...,\lambda_k$ are distinct eigenvalues of a matrix $A$ with multiplicity $p_1,p_2,...,p_k$ respectively, then the set $\{\lambda_1^{p_1},\lambda_2^{p_2},...,\lambda_k^{p_k}\}$ is defined as the spectrum of $A.$ \\
For a simple graph $\Gamma$ with $k$ vertices $\{v_1,v_2,...,v_k\},$ the adjacency matrix $A(\Gamma)=[a_{ij}]$ is a square matrix of size $k$ defined as $a_{ij}=\begin{cases}
 0 & \text{ if } \hspace{0.2cm}   \text{ there is no edge between $v_i$ and $v_j$ }\\
1 & \text{ if }  \hspace{0.2cm}  \text{there is an edge between $v_i$ and $v_j.$ }
\end{cases}$ \\
The degree matrix $D(\Gamma)=[d_{ij}]$ is a diagonal matrix defined as $d_{ii}=\text{deg}(v_i).$ The Laplacian matrix for $\Gamma$ is defined as $L(\Gamma)=D(\Gamma)-A(\Gamma)$ and signless Laplacian matrix for $\Gamma$ is defined as $Q(\Gamma)=D(\Gamma)+A(\Gamma).$ Now we define normalized Laplacian matrix $\mathcal{L}(\Gamma)=l_{ij},$ where 
\begin{center}
$l_{ij}=\begin{cases}
1, &  \hspace{0.2cm}   \text{ $i=j,deg(v_i)\neq 0$}\\
 \frac{-1}{\sqrt{deg(v_i)deg(v_j)}}, &   \hspace{0.2cm}  \text{$v_iv_j\in E(\Gamma)$ }\\
 0, &  \hspace{0.2cm}   \text{otherwise,}
\end{cases}$ 
\end{center}
where $E(\Gamma)$ denotes the set of edges of the graph $\Gamma.$ 

Recall, $\zeta(G,G)$ denotes the commuting graph of group $G$ of size $p^n$ with $|Z(G)|=p^{n-2}.$ Using Proposition \ref{structure of graph}, we get
$$\zeta(G,G) =\begin{cases}
      K_{p^{n-2}}, & V=\Omega_1,\\
      K_{p^{n-1}-p^{n-2}}, & V=B_k, (1\leq k\leq p+1).
  \end{cases}$$
  Let $\mathcal{G}_1=K_{p^{n-2}}$ and $\mathcal{G}_2=(p+1).K_{p^{n-1}-p^{n-2}}.$ Let $S_1$ denotes the star graph where $\mathcal{G}_1$ is the central vertex and $\mathcal{G}_2$ is connected to $\mathcal{G}_1$. Now  we find $\Spec(K_n),\Spec_Q(K_n)$ and $\Spec_\mathcal{L}(K_n)$, which are spectrum of adjacency matrix, signless Laplacian matrix and normalized Laplacian matrix. Using \cite[Proposition 1.4.1]{book}, we get
  \begin{align*}
  \Spec(K_n) &=\{(-1)^{n-1},(n-1)^{1}\},\\
   \Spec(mK_n) &=\{(-1)^{m(n-1)},(n-1)^{m}\}.
   \end{align*}
   We know that signless Laplacian matrix
   $Q(K_n)=D(K_n)+A(K_n),$ where $D(K_{n})=(n-1)I$ is the degree matrix of $K_n,$ and  $A(K_n)$ is the adjacency matrix of $K_n.$ Now, it is easy to compute that 
   
 \begin{align*}
   \Spec_Q(K_n) &=\{2(n-1),(n-2)^{n -1}\}, \\
   \Spec_Q(mK_n) &=\{2(n-1)^m,(n-2)^{m(n -1)}\}.
   \end{align*}

Using \cite[Section 1.2]{Chung:1997}, we get the normalized Laplacian matrix 
$\mathcal{L}(K_n)=D(K_n)^{-\frac{1}{2}}L(K_n)D(K_n)^{-\frac{1}{2}},$ where
$L(K_n)=D(K_n)-A(K_n).$ Simplifying, we get
   $$ \mathcal{L}(K_n) =I-D(K_n)^{-\frac{1}{2}}A(K_n)D(K_n)^{-\frac{1}{2}}
    =I-\left(\frac{1}{n-1}\right)A(K_n).$$ Now, we compute
(
   
    \begin{align*}
   \Spec(\mathcal{L}(K_n)&=\left\{0^1,\left(\frac{n}{n-1}\right)^{n-1}\right\},\\\Spec_\mathcal{L}(K_n) &=\left\{0^1,\left(\frac{n}{n-1}\right)^{(n -1)}\right\},\\
   \Spec_\mathcal{L}(mK_n) &=\left\{0^m,\left(\frac{n}{n-1}\right)^{m(n -1)}\right\}.
   \end{align*}

Let $H$ be a  regular graph with $V(H)=\{1,2,...,k\}$ and $\mathcal{G}_i$ be a $r_i$-regular graph of order $n_i$ where $n_i=|V(\mathcal{G}_i)|,1\leq i\leq k.$ For the graph $H,~N_H(i)$ denotes the neighbors of $i.$ Define $N_i=\sum\limits _{j\in N_H(i)}n_j$ and from \cite[Section 2]{Chang-Xiang}, $C_Q(H)=c_{ij},$ where
$$c_{ij} =\begin{cases}
      2r_i+N_i, & i=j,\\
     \sqrt{n_in_j}, & ij\in E(H),\\
     0, & \text{otherwise}
  \end{cases}$$
With same notations as above, now we consider the graph $\zeta(G,G),$ the commuting graph of group $G$ of size $p^n$ and center size $p^{n-2}.$ We get 
\begin{align*}
(n_1,n_2)=(p^{n-2},(p+1)(p^{n-1}-p^{n-2})),\\
(r_1,r_2)=(p^{n-2}-1,p^{n-1}-p^{n-2}-1).
\end{align*}
From definition of star graph , $N_1=n_2\text{ (all neighbors of $N_1$) },N_2=n_1\text{ (all neighbors of $N_2$). }$\\
$(N_1,N_2)=((p^n-p^{n-2}),p^{n-2}),$ we compute
      
       \[C_Q(S_1)=
\begin{pmatrix}
p^{n-2}+p^n-2 & p^{n-2}\sqrt{p^2-1} \\
p^{n-2}\sqrt{p^2-1} & 2p^{n-1}-p^{n-2}-2
\end{pmatrix}
\]

\begin{theorem} \label{Theorem 6.1}
     Let $\Gamma=\zeta(G,G).$ Then $\Spec_Q(\Gamma)=\Spec(C_Q(S_1))\cup A,$ where $$A=\{(p^{n}-2)^{p^{n-2}-1},(2p^{n-1}-p^{n-2}-2)^p,(p^{n-1}-2)^{p^n-p^{n-2}-p-1}\}.$$
\end{theorem}
     \begin{proof}
       Using \cite[Theorem 2.1]{Chang-Xiang},   we get $$\Spec_Q(\Gamma)=\bigcup\limits_{i=1}^{2}\left(N_i+\left(\Spec_Q(\mathcal{G}_i)\setminus \{2r_i\}\right)\right)\cup \Spec(C_Q(S_1)).$$
        We compute 
         \begin{align*}
           N_1+\left(\Spec_Q(K_{p^{n-2}}\setminus \{2r_1\}\right) &=
           (p^n-p^{n-2})+\{2(p^{n-2}-1),(p^{n-2}-2)^{p^{n-2}-1}\} \setminus\{2(p^{n-2}-1)\}\\
           &=\{(p^{n-2}-2)^{p^{n-2}-1}\}.\\   
         \end{align*}

Similarly, we get $$N_2+\left(\Spec_Q(\mathcal{G}_2)\setminus \{2r_2\}\right)=\{(2p^{n-1}-p^{n-2}-2)^p,(p^{n-1}-2)^{p^n-p^{n-2}-p-1}\}.$$
Combining this, we get
$\Spec_Q(\Gamma)=\Spec(C_Q(S_1))\cup A,$ where  $$A=\{(p^{n}-2)^{p^{n-2}-1},(2p^{n-1}-p^{n-2}-2)^p,(p^{n-1}-2)^{p^n-p^{n-2}-p-1}\}.$$
\end{proof}

\begin{example}
    We compute the signless Laplacian spectrum of commuting graph of group $G$ of size $3^5$ with $|Z(G)|=3^3.$ In this case,  $\mathcal{G}_1=K_{27},~ \mathcal{G}_2=4K_{54},$ and $\Gamma=K_{27}\vee 4K_{54}.$ We compute 
    \[C_Q(S_1)=
\begin{pmatrix}
268 & 27\sqrt{8} \\
27\sqrt{8}& 133
\end{pmatrix}
\]
Now, using Theorem \ref{Theorem 6.1}, we get 
\begin{align*}
  \Spec_Q(\Gamma)&=\{ 241^{26},133^3,79^{212} \}\cup\Spec(C_Q(S_1))\\
&=\{241^{26},133^3,79^{212},98.577,302.423\}.  
\end{align*}

\end{example}

     \begin{theorem} \label{Theorem 6.2} 
         Let $\Gamma=\zeta(G,G).$ Then the normalized Laplacian Spectrum $\Spec_\mathcal{L}(\Gamma)$ is
         $$\left\{ 0, \frac{p^n(p^{n-1}+p^{n-2}-p^{n-3}-1)}{(p^n-1)(p^{n-1}-1)} , \left( \frac{p^n}{p^n-1} \right) ^{p^{n-2}-1}, \left( \frac{p^{n-1}}{p^{n-1}-1} \right)^{p^{n}-p^{n-2}-p-1 }, \left( \frac{p^{n-2}}{p^{n-1}-1} \right)^{p } \right\}.$$
     \end{theorem}
     \begin{proof}
        With same notations as discussed above, using \cite[Theorem 3.1]{Chang-Xiang}, we get
        $$\Spec_\mathcal{L}(\Gamma)=\bigcup\limits_{i=1}^{2}\left(\frac{N_i}{r_i+N_i}+\frac{r_i}{r_i+N_i}(\Spec_\mathcal{L}(\mathcal{G}_i)\setminus\{0\})\right)\bigcup \Spec(C_\mathcal{L}(S_1))$$ and using \cite[Section 3]{Chang-Xiang}, we get $C_{\mathcal{L}}(S_1)=(c_{ij}),$ where
        $$c_{ij} =\begin{cases}
      \frac{N_i}{r_i+N_i}, & i=j,deg(v_i)\neq 0\\
     -\sqrt{\frac{n_in_j}{(r_i+N_i)(r_j+N_j)}}, & ij\in E(H),\\
     0, & \text{otherwise.}
  \end{cases}$$
We compute \[C_\mathcal{L}(S_1)=
\begin{pmatrix}
\frac{p^n-p^{n-2}}{p^n-1} & -\sqrt{\frac{p^{n-2}(p^n-p^{n-2})}{(p^n-1)(p^{n-1}-1)}} \\
-\sqrt{\frac{p^{n-2}(p^n-p^{n-2})}{(p^n-1)(p^{n-1}-1)}} & \frac{p^{n-2}}{p^{n-1}-1}
\end{pmatrix}
\]
We notice that the determinant of the matrix $C_\mathcal{L}(S_1)$ is zero. Therefore, we get
$$\Spec(C_\mathcal{L}(S_1))=\left\{0, \frac{p^n(p^{n-1}+p^{n-2}-p^{n-3}-1)}{(p^n-1)(p^{n-1}-1)} \right\}.$$
We know that,

\begin{align*}
 \Spec_\mathcal{L}(K_{p^{n-2}}) &=\left\{\left(\frac{p^{n-2}}{p^{n-2}-1}\right)^{p^{n-2}-1},0\right\},\\
\Spec_\mathcal{L}((p+1)K_{p^{n-1}-p^{n-2}}) &=\left\{\left(\frac{p^{n-1}-p^{n-2}}{p^{n-1}-p^{n-2}-1}\right)^{p^n-p^{n-2}-p-1},0^{p+1}\right\}. 
\end{align*}

Using this, we compute 
\begin{align*}
    \frac{N_1}{r_1+N_1}+\frac{r_1}{r_1+N_1}\left(\Spec_\mathcal{L}(\mathcal{G}_1)\setminus \{0\}\right) &=\left(\frac{p^n-p^{n-2}}{p^n-1
}\right)+\frac{p^{n-2}-1}{p^n-1}\left(\left\{\left(\frac{p^{n-2}}{p^{n-2}-1\
}\right)^{p^{n-2}-1},0\right\}\setminus\left\{0\right\}\right)\\
&= \left\{ \left( \frac{p^n}{p^n-1} \right) ^{p^{n-2}-1}\right\}.
\end{align*}

Similarly, we compute 
$$\frac{N_2}{r_2+N_2}+\frac{r_2}{r_2+N_2}\left(\Spec_\mathcal{L}(\mathcal{G}_2)\setminus \{0\}\right) = \left\{ \left( \frac{p^{n-1}}{p^{n-1}-1} \right)^{p^{n}-p^{n-2}-p-1 }, \left( \frac{p^{n-2}}{p^{n-1}-1} \right)^{p }\right\}.$$
Combining the above, we get the result.
\end{proof}

\begin{example}
We compute the normalized Laplacian spectrum of commuting graph of group $G$ of size $3^5$ with $|Z(G)|=3^3.$ In this case $\mathcal{G}_1=K_{27},~ \mathcal{G}_2=4K_{54},~ \text{and}~\Gamma=K_{27}\vee4K_{54}.$ 
Using Theorem \ref{Theorem 6.2}, we get\\
$$\Spec_\mathcal{L}(\Gamma)=\left\{\ 0,\frac{11907}{9680}, \left(\frac{243}{242}\right)^{26},\left(\frac{81}{80}\right)^{212},\left(\frac{27}{80}\right)^3 \right\}.$$

\end{example}

\printbibliography
\end{document}